\documentclass[reqno, 11pt,]{amsart}
\usepackage{amsmath, amsthm, amssymb}
\newtheorem{thm}{Theorem}[section]
\newtheorem{cor}[thm]{Corollary}
\newtheorem{lem}[thm]{Lemma}

\newtheorem{defn}[thm]{Definition}
\theoremstyle{remark}

\newcommand{\eps}{\varepsilon}

\newcommand{\Qss}{Q^{\ast\ast}}

\newcommand{\Ta}{T_\alpha}
\newcommand{\Ts}{T^\ast}

\newcommand{\Tas}{\Ts_\alpha}
\newcommand{\xs}{x^\ast}

\newcommand{\Xs}{X^\ast}

\newcommand{\yi}{y_i}
\newcommand{\ys}{y^\ast}

\newcommand{\Ys}{{Y^\ast}}
\newcommand{\Yss}{Y^{\ast\ast}}
\newcommand{\Ysss}{Y^{\ast\ast\ast}}

\newcommand{\fwsxsy}{\mathcal F_{w^\ast}(\Xs,Y)}

\newcommand{\fwsxsyss}{\mathcal F_{w^\ast}(\Xs,\Yss)}

\newcommand{\ixys}{\mathcal I(X,\Ys)}
\newcommand{\ixysss}{\mathcal I(X,\Ysss)}

\newcommand{\lxsy}{\mathcal L(\Xs,Y)}

\newcommand{\lysx}{\mathcal L(\Ys,X)}

\newcommand{\sumin}{\sum_{i=1}^n}

\newcommand{\wast}{weak$^\ast$}

\begin{document}

\title[Principle of local reflexivity respecting subspaces]{Principle of local reflexivity\\ respecting subspaces}%
\author{Eve Oja}%
\address{Faculty of Mathematics and Computer Science\\ University of Tartu\\ J. Liivi 2\\ 50409 Tartu, Estonia; Estonian Academy of Sciences\\ Kohtu 6\\ 10130 Tallinn, Estonia}
\email{eve.oja@ut.ee}%
\thanks{This research was partially supported by Estonian Science Foundation Grant 8976 and Estonian Targeted Financing Project SF0180039s08}%

\begin{abstract}

We obtain a strengthening of the principle of local reflexivity in a general form. The added strength makes local reflexivity operators respect given subspaces. Applications are given to bounded approximation properties of pairs, consisting of a Banach space and its subspace.
\end{abstract}
\maketitle

\section{Introduction and the Main Results}

The principle of local reflexivity (PLR) is a powerful tool in the theory of Banach spaces and its applications. The PLR shows that the bidual $X^{**}$ of a Banach space $X$ is ``locally" almost the same as the space $X$ itself.

\begin{thm}[{The PLR, see \cite{LR} and \cite{JRZ} or, e.g., \cite[p. 53]{JL}}]
Let $X$ be a Banach space. If $E$ and $F$ are finite-dimensional subspaces of $X^{**}$ and $X^*$, respectively, and $\varepsilon>0$, then there exists a one-to-one linear operator $T:E\to X$ such that
\begin{itemize}
\item[${\rm 1^\circ}$]\ $\| T\|, \|T^{-1}\|<1+\varepsilon$,
\item[${\rm 2^\circ}$]\ $y^*(Tx^{**})=x^{**}(y^*)$\quad{\it for\ all}\quad $x^{**}\in E$\quad {\it and}\quad $y^*\in F$,
\item[${\rm 3^\circ}$]\ $Tx^{**}=x^{**}$\quad {\it for\ all}\quad $x^{**}\in E\cap X$.
\end{itemize}
\end{thm}

The PLR was discovered by Lindenstrauss and Rosenthal \cite{LR} in 1969. It was improved by Johnson, Rosenthal, and  Zippin \cite{JRZ} in 1971. Since then, many new proofs, refinements, and generalizations of the PLR have been given in the literature (see, e.g., \cite{B2} and \cite{OP} for results and references).

Recently, the concept of the bounded approximation property of pairs was introduced and studied by Figiel, Johnson, and Pe{\l}czy\'{n}ski in the important paper \cite{FJP}. This concept involves finite-rank operators fixing a subspace in a given Banach space (see Section 4). Studies in \cite{FJP} and \cite{OT} seem to indicate a need for some kind of PLR that would respect subspaces. In the present paper, we shall establish such versions of PLR, see Theorems 1.2 and 1.3 below.

First, let us fix some (standard) notation. For Banach spaces $X$ and $Y$, both real or both complex, we denote by $\mathcal{L}(X,Y)$ the Banach space of all bounded linear operators from $X$ to $Y$. And let $\mathcal{F}(X,Y)$ denote its subspace of finite-rank operators. If $U$ is a subspace of $X$, then $U^\bot$ is its annihilator in the dual space $X^*$, i.e., $U^\bot=\{x^*\in X^*: x^*(x)=0\ \forall x\in U\}$.

\begin{thm}[PLR respecting subspaces]
Let $X$ and $Y$ be Banach spaces, and let $U$ and $V$ be closed subspaces of $X$ and $Y$, respectively. Let $S\in \mathcal{F}(X^{**}, Y^{**})$ satisfy $S(U^{\bot\bot})\subset V^{\bot\bot}$. If $E$ and $F$ are finite-dimensional subspaces of $X^{**}$ and $Y^*$, respectively, and $\varepsilon>0$, then there exists an operator $T\in \mathcal{F}(X,Y)$ satisfying $T(U)\subset V$ such that
\begin{itemize}
\item[${\rm 1^\circ}$]\ $\bigl| \| T\|-\| S\|\bigr|<\varepsilon$,
\item[${\rm 2^\circ}$]\ $x^{**}(T^*y^*)=(Sx^{**})(y^*)$\quad {\it for\ all}\quad $x^{**}\in E$\quad {\it and\ all}\quad $y^*\in F$,
\item[${\rm 3^\circ}$]\ $T^{**}x^{**}=Sx^{**}$\quad {\it for\ all\ those}\quad $x^{**}\in E$\quad {\it for\ which}\quad $Sx^{**}\in Y$.
\end{itemize}
Moreover, if the restriction $S|_E$ is one-to-one, then also $T^{**}|_E$ is, and

\begin{itemize}
\item[${\rm 1^{\circ\circ}}$]\ $\|(T^{**}|_E)^{-1}\| < \|(S|_E)^{-1}\| + \varepsilon$ .
\end{itemize}
In the special case, when $X=Y$ and $S$ is a projection, also $T$ is a projection.
\end{thm}

As an illustration how to apply Theorem 1.2, let us look at the probably most well-known variant of Bellenot's PLR \cite{Be}.
This PLR asserts the same as Theorem 1.1 but, additionally, a closed subspace $W$ of $X$ is given and, correspondingly, the assertion $T(E\cap W^{\bot\bot})\subset W$ is obtained. (Thus, for $W=\{ 0\}$, Bellenot's PLR reduces to Theorem 1.1.)

Bellenot's PLR (and, in particular, Theorem 1.1) follows immediately from Theorem 1.2. Indeed, one only has to take, in Theorem 1.2, a Banach space $X$ instead of $Y$, $E$ instead of $X$, $W$ instead of $V$, and $U=E\cap W^{\bot\bot}$. If $S:E\to X^{**}$ is the identity embedding, then $S(U^{\bot\bot})=U\subset W^{\bot\bot}$, and $T:E\to X$ given by Theorem 1.2 clearly satisfies conditions ${\rm 1^\circ}$--${\rm 3^\circ}$ of Theorem 1.1, and also $T(E\cap W^{\bot\bot})\subset W$.

Theorem 1.2 will be proven in Section 3. Its proof will be based on Theorem 1.3 below, which is another main result of the present paper. It could be called ``a theorem on locally conjugate operators who respect subspaces".

\begin{thm}
Let $X$ and $Y$ be Banach spaces, and let $U$ and $V$ be closed subspaces of $X$ and $Y$, respectively. Let $S\in \mathcal{F}(Y^*,X^*)$ satisfy $S(V^\bot)\subset U^\bot$. If $F$ is a finite-dimensional subspace of $Y^*$ and $\varepsilon>0$, then there exists $T\in\mathcal{F}(X,Y)$ satisfying $T(U)\subset V$ such that
\begin{itemize}
\item[${\rm 1^\circ}$]\ $\bigl| \| T\|-\| S\|\bigr|<\varepsilon$,
\item[${\rm 2^\circ}$]\ $T^*y^*=Sy^*$\quad {\it for\ all}\quad $y^*\in F$,
\item[${\rm 2^{\circ\circ}}$]\ ${\rm ran}\ T^*= {\rm ran}\ S$,
\item[${\rm 3^\circ}$]\ $T^{**}x^{**}=S^*x^{**}$\quad {\it for\ all\ those}\quad $x^{**}\in X^{**}$\quad {\it for\ which}\quad $S^*x^{**}\in Y$.
\end{itemize}
In the special case, when $X=Y$ and $S$ is a projection, also $T$ is a projection.
\end{thm}

The special case of Theorem 1.3 without requirements concerning subspaces (i.e., the case when $U=\{0\}$ and $V=\{0\}$) is essentially due to Johnson, Rosenthal, and Zippin \cite[Lemma 3.1 and Corollary 3.2]{JRZ} (where $X$ was assumed to be finite-dimensional; see \cite[Theorem 2.5]{OP} for the infinite-dimensional case and an easier proof). A version of this special case for positive finite-rank operators between Banach lattices was obtained in \cite[Theorem 5.6]{LiO}. Theorem 1.3 will also be proven in Section 3.

Section 4 contains applications of Theorems 1.2 and 1.3 to bounded approximation properties of pairs. These are easy and clear proofs of some classical theorems on approximation properties in their new general context of pairs. For instance, we show that the pair $(X,Y)$, consisting of a Banach space $X$ and its closed subspace $Y$, has the $\pi_{\lambda\mu+\lambda+\mu}$-duality property (see Definition 4.2) whenever the pair $(X^*, Y^\bot)$ has the $\pi_\lambda$-property and $(X,Y)$ has the $\mu$-bounded approximation property. This refines a well-known result from \cite{JRZ}.

The very basic idea of our proofs consists in using Grothendieck's description of the dual of the space of  (weak*-to-weak) continuous finite-rank operators as integral operators (see Section 2). This basic idea comes from our paper \cite[Corollary 2.3]{O1}, where it was used to give an alternative PLR-free proof of Johnson's theorem \cite{J} (see, e.g., \cite[Proposition 3.5]{C}) asserting that the bounded approximation property of a dual space can always be given with conjugate operators. The idea was then used in our papers \cite{LO}, \cite{O3}, \cite{OP}, and \cite{OT}.

In \cite{OP}, departing from Grothendieck's description, we gave a clear proof of the PLR in a general form (encompassing, e.g., Theorem 1.1 together with  its refined version due to Behrends \cite{B1}). However, for instance, Bellenot's PLR (Theorem 1.2 is its strengthening) was not approached in \cite{OP}. To encompass the general case when there are given subspaces to be respected (as in Theorems 1.2 and 1.3), we need further develop the basic idea together with methods from \cite{OP} and \cite{OT}; see Section 2. In doing so, we shall follow the general scheme from \cite{OP}. (Quite natural, since we continue to believe that \cite{OP} gives the best proof of the PLR. We even considered titling \cite{OP} as ``The principle of local reflexivity done right", but we did not dare to paraphrase \cite{A}.)

\section{The Main Lemma}

Let $X$ and $Y$ be Banach spaces. Denote by $\fwsxsy$
the subspace of $\lxsy$ consisting of \wast-to-weak continuous finite-rank operators.
By a well-known description, due to Grothendieck \cite[Chapter I, pp. 124--125]{G}
(see, e.g., \cite[pp. 231--232]{DU} or \cite[p. 58]{R}), $(\fwsxsy)^\ast=\ixys$,
the Banach space of integral operators (equipped with their integral norms).

Recall that $\fwsxsy$ is algebraically the same as the algebraic tensor product $X\otimes Y$, with the rank one operator $x\otimes y: x^* \mapsto x^*(x)y$, $x^*\in X^*$, corresponding to the elementary tensor $x\otimes y$.
Thus, any $S\in\fwsxsy$ can be written as $S=\sumin x_i\otimes\yi$ for some $x_i\in X$ and $\yi\in Y$.
The above-mentioned identification $(\fwsxsy)^\ast=\ixys$ is realized via the duality
\[
\Bigl\langle A,\sumin x_i\otimes\yi\Bigr\rangle=\sumin(A x_i)(\yi).
\]

Our Main Lemma is Lemma 2.1. It reduces to \cite[Lemma 1.2]{OP} in the particular case when $G=\{0\}$ and $V=\{0\}$.

\begin{lem}[Main Lemma]
Let $X$ and $Y$ be Banach spaces, let $G$ be a linear subspace of $X^*$, and $V$ a closed subspace of $Y$. Let $T\in \mathcal{F}_{w^*}(X^*, Y^{**})$ satisfy $T(G)\subset V^{\bot\bot}$. If $F$ is a finite-dimensional subspace of $Y^*$, then there exists a net $(T_\alpha)\subset \mathcal{F}_{w^*}(X^*, Y)$ satisfying $T_\alpha(G)\subset V$ for all $\alpha$ such that
\begin{itemize}
\item[${\rm 1^\circ}$]\ $\|T_\alpha\|\to\| T\|$,
\item[${\rm 2^\circ}$]\ $T^*_\alpha y^*\to T^*y^*$\ {\it for\ all}\ $y^*\in Y^*$,\ {\it in\ particular,}\ $T^*_\alpha y^*=T^*y^*$\ {\it for\ all}\ $\alpha$\ {\it and\ for\ all }\ $y^*\in F$,
\item[${\rm 3^\circ}$]\ $T_\alpha x^*=Tx^*$\ {\it for\ all}\ $\alpha$\ {\it and\ for\ all\ those}\ $x^*\in X^*$\ {\it for\ which}\ $Tx^*\in Y$.
\end{itemize}
\end{lem}

In the proof of Lemma 2.1, we shall occasionally need the following auxiliary result.

\begin{lem}[see {\cite[Lemma 4.4]{OT}}]
Let $X$ be a locally convex Hausdorff space. Let $Y$ be a closed subspace and $F$ a finite-dimensional subspace of $X$. Then there exists a continuous linear projection $P$ on $X$ such that ${\rm ran}~P=F$ and $P(Y)\subset Y$.
\end{lem}

\vskip\baselineskip
{\it Proof of Lemma 2.1.} (a) We start by using the canonical identifications $(\fwsxsy)^\ast=\ixys$ and $(\fwsxsyss)^\ast=\ixysss$ as in \cite[Lemma 1.1]{OP}. But we continue by developing \cite[proofs of Lemma 2.2 and Theorem 1.2]{OT} as follows.

Define $J:\mathcal{I}(X,Y^*)\to \mathcal{I}(X,Y^{***})$ by $JA=j_{Y^*}A$, $A\in \mathcal{I}(X,Y^*)$, where $j_{Y^*}:Y^*\to Y^{***}$ is the canonical embedding. Then $J$ is well known to be isometric (see, e.g., \cite[p. 65]{R}).

Denote $\mathcal{R}=\{R\in \mathcal{F}_{w^*}(X^*, Y):R(G)\subset V\}$ and $\mathcal{S}=\{ S\in \mathcal{F}_{w^*}(X^*,Y^{**}):S(G)\subset V^{\bot\bot}\}$. Consider $\mathcal{R}^\bot$ and $\mathcal{S}^\bot$ as subspaces of $\mathcal{I}(X,Y^*)$ and $\mathcal{I}(X,Y^{***})$, respectively. Similarly to the quite technical but elementary proof of \cite[Lemma 2.2]{OT}, we can show that $J(\mathcal{R}^\bot)\subset\mathcal{S}^\bot$. This implies that the operator
\[
\bar{J}: \mathcal{I}(X,Y^*) /\mathcal{R}^{\perp} \to \mathcal{I}(X,Y^{***})/ \mathcal{S}^{\perp},
\]
given by
\[
\bar{J}(A+\mathcal{R}^{\perp})= JA+\mathcal{S}^{\perp}, \ A \in \mathcal{I}(X,Y^*),
\]
is well defined.

Since $\mathcal{R}$ and $\mathcal{S}$ are linear subspaces of $\mathcal{F}_{w^*}(X^*,Y)$ and $\mathcal{F}_{w^*}(X^*,Y^{**})$, their duals $\mathcal{R}^*$ and $\mathcal{S}^*$ are canonically isometrically isomorphic to $(\mathcal{F}_{w^*}(X^*, Y))^*/\mathcal{R}^\bot$ and $(\mathcal{F}_{w^*}(X^*,Y^{**})^*/\mathcal{S}^\bot$. Hence, under the canonical identifications, $\mathcal{R}^*$ and $\mathcal{S}^*$ are isometrically isomorphic to $\mathcal{I}(X,Y^*)/\mathcal{R}^\bot$ and $\mathcal{I}(X,Y^{***})/\mathcal{S}^\bot$, respectively. Let $s:\mathcal{R}^*\to\mathcal{I}(X,Y^*)/\mathcal{R}^\bot$ and $t:\mathcal{I}(X,Y^{***})/\mathcal{S}^\bot\to\mathcal{S}^*$ denote the corresponding isometries.

Define $\Phi:\mathcal{R}^*\to\mathcal{S}^*$ by $\Phi=t\bar{J}s$.
Then clearly $\|\Phi\|\leq1$ and thus
$\Phi^\ast(T)\in\|T\|B_{\mathcal{R}^{**}}$.
By Goldstine's theorem,
there is a net $(\Ta)\subset\mathcal{R}$, meaning that $T_\alpha(G)\subset V$ for all $\alpha$, converging weak$^\ast$ to $\Phi^\ast(T)$ such that $\sup_\alpha\|\Ta\|\leq\|T\|$.
In particular, for all $\xs\in\Xs$ and $\ys\in\Ys$, considering $x^*\otimes y^*$ in $\mathcal{R}^*$ and $x^*\otimes j_{Y^*}y^*$ in $\mathcal{S}^*$, we have
\begin{align*}
\xs(\Tas\ys)
&=\langle\xs\otimes\ys,\Ta\rangle\to\bigl(\Phi^\ast(T)\bigr)(\xs\otimes\ys)=\bigl(\Phi(\xs\otimes\ys)\bigr)(T)\\
&=\langle\xs\otimes j_{\Ys}\ys, T\rangle=\xs(\Ts\ys).
\end{align*}
This means that $\Tas\to\Ts|_\Ys$ in the weak operator topology of $\lysx$.
Now, using a convex combination argument as in the proof of \cite[Lemma 1.1]{OP}, we may assume that our $(T_\alpha)$ also satisfies ${\rm 1^\circ}$, $T_\alpha^*y^*\to T^*y^*$ for all $y^*\in Y^*$, and $T_\alpha x^*\to Tx^*$ for all those $x^*\in X^*$ for which $Tx^*\in Y$.

\vskip\baselineskip
(b) Following the scheme of the proof of \cite[Lemma 1.2]{OP}, we  first prove the Main Lemma in the particular case when $X$ is finite-dimensional. Then $G$ is a closed subspace of $X$. Endowing $Y^*$ with its  weak* topology, we have that $V^\bot$ is a closed subspace of $Y^*$. Using now Lemma 2.2, we can choose projections $P\in\mathcal{L}(X^*)$ and $Q\in\mathcal{L}(Y^*)$, $Q$ being weak*-to-weak* continuous, such that ${\rm ran}~P=T^{-1}(Y), P(G)\subset G$, ${\rm ran}~Q=F$, and $Q(V^\bot)\subset V^\bot$. Then there also exists a projection $R\in \mathcal{L}(Y)$ such that $R^*=Q$ and $R(V)\subset V$.

Applying the perturbation argument from \cite[proof of Lemma 1.2]{OP}, we denote, for all $\alpha$,
$$
S_\alpha=TP+T_\alpha(I-P)-Q^*(T_\alpha-T)(I-P).
$$
Then ${\rm ran}~S_\alpha\subset Y$, because ${\rm ran}~TP\subset Y$ (recall that ${\rm ran}~P=T^{-1}(Y)$) and ${\rm ran}~Q^*={\rm ran}~R^{**}\subset Y$ (since $R$ is of finite rank). Hence, $S_\alpha\in\mathcal{F}(X^*,Y)=\mathcal{F}_{w^*}(X^*,Y)$, because $\dim X<\infty$.

Observing that
$$
S_\alpha=TP+T_\alpha-T_\alpha P-RT_\alpha+RT_\alpha P+R^{**}T-R^{**}TP,
$$
we get that $S_\alpha(G)\subset V$, because $P(G)\subset G$, $(TP)(G)\subset Y\cap V^{\bot\bot}=\overline{V}^{weak}=\overline{V}=V$, $T_\alpha(G)\subset V$, $R(V)\subset V$, and $(R^{**}T)(G)\subset Y\cap V^{\bot\bot}=V$.

Moreover, in the proof of \cite[Lemma 1.2]{OP}, it is established that $S_\alpha^*y^*=T^*y^*$ for all $y^*\in F$ and $S_\alpha x^*=Tx^*$ for all $x^*\in X^*$ for which $Tx^*\in Y$. Since there is also shown that $\| S_\alpha-T_\alpha\|\to 0$, we infer that $\| S_\alpha\|\to\| T\|$ and $S_\alpha^*y^*\to T^*y^*$ for all $y^*\in Y^*$. In conclusion, $(S_\alpha)$ is the desired net, and the Main Lemma holds in the particular case when $X$ is of finite dimension.

\vskip\baselineskip
(c) Let us look at the general case. Since $T\in\mathcal{F}_{w^*}(X^*, Y^{**})$, we have $Z:={\rm ran}~T^*\subset X$. Denote the identity embedding by $j:Z\to X$ and let $S\in\mathcal{F}_{w^*}(Z^*, Y^{**})$ satisfy $T=Sj^*$. In fact, since $j^*$ is a restriction mapping, i.e., $j^*x^*=x^*|_Z$, $x^*\in X^*$, and $Z={\rm ran}~T^*$, one clearly may define $Sz^*=Tx^*$, $z^*\in Z^*$, where $x^*\in X^*$ is any functional satisfying $z^*=j^*x^*$. Applying Part (b) to $j^*(G)\subset Z^*$ and $S$, we obtain $(S_\alpha)\subset \mathcal{F}_{w^*}(Z^*,Y)$ satisfying $(S_\alpha j^*)(G)\subset V$ for all $\alpha$ such that $\| S_\alpha j^*\|=\| S_\alpha\|\to\| S\|=\| Sj^*\|=\| T\|$ (because $j^*(B_{X^*})=B_{Z^*}$), $S_\alpha^*y^*\to S^*y^*$ for all $y^*\in Y^*$ with $S_\alpha^*y^*= S^*y^*$ for $y^*\in F$. And, finally, $S_\alpha z^*=Sz^*$ for all $z^*\in Z^*$ for which $Sz^*\in Y$, meaning that $S_\alpha j^*x^*=Sj^*x^*=Tx^*$ for all $x^*\in X^*$ for which $Sj^*x^*=Tx^*\in Y$. Hence, $(T_\alpha):=(S_\alpha j^*)\subset\mathcal{F}_{w^*}(X^*,Y)$ is the desired net.
\begin{flushright}
$\square$
\end{flushright}

\section{Proofs of Theorems 1.3 and 1.2}

As was mentioned in the Introduction, Theorem 1.3 is a stronger form of \cite[Theorem 2.5]{OP}. The added strength makes the local reflexivity operator respect given subspaces.
Theorem 3.1 below is a net version of Theorem 1.3. {\it Theorem} 1.3 {\it immediately follows from Theorem} 3.1 by choosing $\alpha$ large enough to have $\bigl|\|\Ta\|-\|S\|\bigr|<\eps$,
and putting $T=\Ta$.

\begin{thm}
Let $X$ and $Y$ be Banach spaces, and let $U$ and $V$ be closed subspaces of $X$ and $Y$, respectively. Let $S\in \mathcal{F}(Y^*,X^*)$ satisfy $S(V^\bot)\subset U^\bot$. If $F$ is a finite-dimensional subspace of $Y^*$, then there exists a net $(T_\alpha)\subset\mathcal{F}(X,Y)$ satisfying $T_\alpha(U)\subset V$ for all $\alpha$ such that
\begin{itemize}
\item[${\rm 1^\circ}$]\ $\| T_\alpha\|\to\| S\|$
\item[${\rm 2^\circ}$]\ ${\rm ran}~T_\alpha^*={\rm ran}~S$\ {\it for\ all}\ $\alpha$
\item[${\rm 3^\circ}$]\ $T_\alpha^*y^*\to Sy^*$\ {\it for\ all}\ $y^*\in Y^*$,\ {\it in\ particular,}\ $T_\alpha^*y^*=Sy^*$\ {\it for\ all}\ $y^*\in F$,
\item[${\rm 4^\circ}$]\ $T_\alpha^{**}x^{**}=S^*x^{**}$\ {\it for\ all}\ $\alpha$\ {\it and\ for\ all\ those}\ $x^{**}\in X^{**}$\ {\it for\ which}\ $S^*x^{**}\in Y$.
\end{itemize}
Moreover, if $X=Y$ and $S$ is a projection, then also the operators $T_\alpha$ are projections.
\end{thm}

\begin{proof}
As in the proof of \cite[Theorem 2.5]{OP}, we start by enlarging $F$, if necessary, so that $S(F)={\rm ran}~S$. Observing that $S^*\in\mathcal{F}_{w^*}(X^{**},Y^{**})$ and $S^*(U^{\bot\bot})\subset V^{\bot\bot}$, we apply Lemma 2.1. Let $(S_\alpha)\subset\mathcal{F}_{w^*}(X^{**},Y)$ be given by Lemma 2.1. Denote $T_\alpha=S_\alpha|_X$. By the proof of \cite[Theorem 2.5]{OP}, $(T_\alpha)$ satisfies conditions ${\rm 1^\circ}$--${\rm4^\circ}$. Since  $S_\alpha(U^{\bot\bot})\subset V$, we also have $T_\alpha(U)\subset V$ for all $\alpha$.

For the ``moreover" part, we start by enlarging $F$, if necessary, so that $F\supset {\rm ran}~S$. Applying what we have already proven, we get a net $(T_\alpha)\subset\mathcal{F}(X,X)$ as above. Using ${\rm 2^\circ}$ and ${\rm 3^\circ}$, it can be easily verified (see the proof of \cite[Theorem 2.5]{OP}) that $T_\alpha^2=T_\alpha$ for all $\alpha$.
\end{proof}

To prove Theorem 1.2, we shall first use Theorem 1.3 twice. Then we shall apply an enlarging argument which is inspired by the classical one, well-known from the proofs of the classical PLR (which is Theorem 1.1 in the Introduction); see, e.g., \cite{D}.

\vskip\baselineskip
{\it Proof of Theorem 1.2.} (a) We are given $S\in\mathcal{F}(X^{**},Y^{**})$ such that $S(U^{\bot\bot})\subset V^{\bot\bot}$. Applying Theorem 1.3 to $S$ yields $R\in\mathcal{F}(Y^*, X^*)$ such that $R(V^\bot)\subset U^\bot$, $\bigl| \| R\|-\| S\| \bigr|<\varepsilon/2$, and $R^*x^{**}=Sx^{**}$ for all $x^{**}\in E$. Applying Theorem 1.3 to $R$ yields $T\in\mathcal{F}(X,Y)$ such that $T(U)\subset V$, $\bigl| \| T\|-\| R\|\bigr|<\varepsilon/2$, $T^*y^*=Ry^*$ for all $y^*\in F$, and $T^{**}x^{**}=R^*x^{**}$ for those $x^{**}\in X^{**}$ for which $R^*x^{**}\in Y$. Hence, conditions ${\rm 1^\circ}$--${\rm 3^\circ}$ clearly hold.

In the special case, when $X=Y$ and $S$ is a projection, also  $R$ and $T$ can be chosen to be projections (see Theorem 1.3).

\vskip\baselineskip
(b) For the ``moreover" part, let us assume that $S|_E$ is one-to-one and denote $\sigma=\| (S|_E)^{-1}\|$. Choose $\delta>0$ such that $\sigma^{-1}-2\delta>0$ and  $(\sigma^{-1}-2\delta)^{-1}<\sigma+\varepsilon$.

Let $Se_1, \ldots , Se_n$, $e_i\in S_E$, be a $\delta$-net for $S(S_E)$. Find $f_i\in S_{Y^*}$ such that
$$
\| Se_i\|-\delta< (Se_i)(f_i),\quad i=1,\ldots , n.
$$
We may assume, by enlarging $F$, if necessary, that $F$ contains $f_i$, $i=1,\ldots , n$. Let $T\in \mathcal{F}(X,Y)$ be an operator given by Part (a).

Denote $t=T^{**}|_E$. We have to prove that $t$ is one-to-one and ${\rm 1^{\circ\circ}}$ holds, i.e., $\| t^{-1}\|<\sigma+\epsilon$.

Fix arbitrarily $e\in S_E$. Then, for all $i=1, \ldots , n$,
$$
\| te\|\geq |(T^{**}e)(f_i)|=|e(T^*f_i)|=|(Se)(f_i)|.
$$
Choosing $e_i\in S_E$ such that $\| Se-Se_i\|<\delta$, we get that $|(Se-Se_i)(f_i)|<\delta$. Therefore,
$$
|(Se)(f_i)|>|(Se_i)(f_i)|-\delta>\| Se_i\|-2\delta\geq\sigma^{-1}-2\delta,
$$
because $1=\| e_i\|=\|(S|_E)^{-1}Se_i\|\leq\sigma\| Se_i\|$. We see that
$$
\| te\|\geq\sigma^{-1}-2\delta\quad\forall e\in S_E.
$$
Hence, $t$ is an invertible operator from $E$ onto ${\rm ran}~t$, and
$$
\| t^{-1}\|\leq (\sigma^{-1}-2\delta)^{-1}<\sigma+\varepsilon.
$$

\begin{flushright}
$\square$
\end{flushright}

\section{Applications to Bounded Approximation Properties of Pairs}

Let $Y$ be a closed subspace of a Banach space $X$. Let $\lambda\geq 1$. The pair $(X,Y)$ is said to have the $\lambda$-{\it bounded approximation property} if for every finite-dimensional subspace $E$ of $X$ and for every  $\varepsilon > 0$ there exists $S \in \mathcal{F}(X):=\mathcal{F}(X,X)$ such that $S(Y) \subset Y$ and $\|S\| \leq \lambda + \varepsilon$, and $Sx=x$ for all $x \in E$. This concept was recently introduced and studied by Figiel, Johnson, and Pe\l czy\'{n}ski in the important paper \cite{FJP}. If $Y=X$ or $Y=\{0\}$, then the $\lambda$-bounded approximation property of the pair $(X,Y)$ is just the classical $\lambda$-bounded approximation property of $X$.

An important theorem due to Johnson \cite{J} (see, e.g., \cite[Proposition 3.5]{C}) asserts that {\it if
$X^{*}$ has the $\lambda$-bounded approximation property, then $X^{*}$ has the $\lambda$-bounded approximation property with conjugate operators.}
This means that the approximating operators $S\in\mathcal{F}(X^*)$ can be chosen to be conjugate.

Johnson's theorem easily follows from the PLR (see, e.g., \cite[Proposition 3.5]{C}), according to which finite-rank operators on a dual space are ``locally conjugate'' (see, e.g., \cite[Theorem 2.5]{OP}). From Theorem 1.3, we immediately get the following extension of Johnson's theorem which was recently established in \cite{OT} relying on the same basic idea, as in the present paper, to use Grothendieck's description of the dual space of the space of finite-rank operators as a space of integral operators.

\begin{cor}[Oja--Treialt]
Let $X$ be a Banach space and let $Y$ be a closed subspace of $X$. If the pair $(X^{*}, Y^{\perp})$ has the $\lambda$-bounded approximation property, then for every finite-dimensional subspace $F$ of $X^*$ and for every $\varepsilon > 0$ there exists $S\in\mathcal{F}(X)$ such that $S(Y)\subset Y$ and $\| S\|\leq\lambda+\varepsilon$, and $S^*x^*=x^*$ for all $x^*\in F$.
\end{cor}

 Applying standard arguments to the operators $S$ from Corollary 4.1, among others, by passing to convex combinations, it follows that there exists a net $(S_\alpha)\subset \mathcal{F}(X)$ such that $S_\alpha(Y)\subset Y$ and $\| S_\alpha\|\leq\lambda$ for all $\alpha$, and $S_\alpha\to I_X$ and $S^*_\alpha\to I_{X^*}$ pointwise. In particular, {\it the pair $(X,Y)$ has the $\lambda$-bounded approximation property.}

Using projections as approximating operators, it is natural to introduce the following concepts.

\begin{defn}
{\rm We say that  a pair $(X,Y)$, where $X$ is a Banach space and $Y$ its closed subspace,  has the $\pi_\lambda$-{\it property} if for every finite-dimensional subspace $E$ of $X$ and for every $\varepsilon>0$ there exists a projection $P\in\mathcal{F}(X)$ such that $P(Y)\subset Y$ and $\| P\|\leq\lambda+\varepsilon$, and $Px=x$ for all $x\in E$. We say that the pair $(X,Y)$ has the $\pi_\lambda$-{\it duality property} if for all finite-dimensional subspaces $E$ of $X$ and $F$ of $X^*$, and for every $\varepsilon>0$ there exists a projection $P\in \mathcal{F}(X)$ such that $P(Y)\subset Y$ and $\| P\|\leq\lambda+\varepsilon$, and $Px=x$ for all $x\in E$ and $P^*x^*=x^*$ for all $x^*\in F$.}
\end{defn}

Again, if $Y=X$ or $Y=\{0\}$, then the $\pi_\lambda$-property of the pair $(X,Y)$ is equivalent, through  the classical perturbation argument due to Johnson, Rosenthal, and Zippin (see \cite[Lemma 2.4]{JRZ} or, e.g., \cite[Lemma 3.2]{C}), to the $\pi_\lambda$-property of $X$. In this special case, the following result is essentially contained in \cite{JRZ}. It is immediate from Theorem 1.3.

\begin{cor}
Let $X$ be a Banach space and let $Y$ be a closed subspace of $X$. If the pair $(X^{*}, Y^{\perp})$ has the $\pi_\lambda$-property, then for every finite-dimensional subspace $F$ of $X^*$ and for every $\varepsilon>0$ there exists a projection $P\in\mathcal{F}(X)$ such that $P(Y)\subset Y$ and $\| P\|\leq\lambda+\varepsilon$, and $P^*x^*=x^*$ for all $x^*\in F$.
\end{cor}

Convex combinations of projections are far from being projections. Therefore, {\it if $(X^*, Y^\bot)$ has the $\pi_\lambda$-property}, standard arguments, applied to the projections from Corollary 4.3, only yield that $(X,Y)$ {\it has the $\lambda$-bounded approximation property}. Even in the classical case, it seems to be unknown whether the $\pi_\lambda$-property of $X^*$ implies or not the $\pi_\lambda$-property of $X$. However, by a famous result of Johnson, Rosenthal, and Zippin \cite[Corollary 4.8]{JRZ}, $X$ has the $\pi_\mu$-property for some $\mu$. And the best known $\mu$ seems to be $\mu=\lambda^2+2\lambda$ (see \cite[p. 618]{S}).
This result is contained in a special case of our following application of Theorem 1.3 (or, more precisely, its Corollaries 4.3 and 4.1).

\begin{thm}
Let $X$ be a Banach space and let $Y$ be a closed subspace of $X$.

{\rm (a)} If $(X^* ,Y^\bot)$ has the $\pi_\lambda$-property and $(X,Y)$ has the $\mu$-bounded approximation property, then the pair $(X,Y)$ has the $\pi_{\lambda\mu+\lambda+\mu}$-duality property.

{\rm (b)} If $(X,Y)$ has the $\pi_\lambda$-property and $(X^*, Y^\bot)$ has the $\mu$-bounded approximation property, then the pair $(X,Y)$ has the $\pi_{\lambda\mu+\lambda+\mu}$-duality property.
\end{thm}

\begin{proof}
Our proof will refine some ideas from \cite{JRZ} and \cite[proof of Theorem 5.6]{C}.

Let $E$ and $F$ be finite-dimensional subspaces of $X$ and $X^*$, respectively, and let $\varepsilon>0$. Choose $\delta>0$ such that $(2+\lambda+\mu)\delta+\delta^2<\varepsilon$.

\vskip\baselineskip
(a) The $\mu$-bounded approximation property of $(X,Y)$ gives us $S\in\mathcal{F}(X)$ such that $S(Y)\subset Y$ and $\| S\|\leq\mu+\delta$, and $Sx=x$ for all $x\in E$. Enlarging $F$, if necessary, we may assume that $F={\rm ran}~  S^*$. By Corollary 4.3, there is a projection $P\in \mathcal{F}(X)$ such that $P(Y)\subset Y$ and $\| P\|\leq\lambda+\delta$, and $P^*x^*=x^*$ for all $x^*\in F={\rm ran}~S^*$. Hence, $P^*S^*x^*=S^*x^*$ for all $x^*\in X^*$, i.e., $P^*S^*=S^*$, meaning that $SP=S$. Then, clearly, $Q:=P+S-PS\in\mathcal{F}(X)$ is a projection, $Q(Y)\subset Y$, $\| Q\|\leq\lambda\mu+\lambda+\mu+\varepsilon$, $Qx=Px+x-x=x$ for all $x\in E$, and $Q^*x^*=x^*+S^*x^*-S^*x^*=x^*$ for all $x^*\in F$.

\vskip\baselineskip
(b) The $\pi_\lambda$-property of $(X,Y)$ gives us a projection $P\in\mathcal{F}(X)$ such that $P(Y)\subset Y$, $\| P\|\leq\lambda+\delta$, and $Px=x$ for all $x\in E$. Enlarging $F$, if necessary, we may assume that $F={\rm ran}~  P^*$. By Corollary 4.1, there is $S\in\mathcal{F}(X)$ such that $S(Y)\subset Y$, $\| S\|\leq\mu+\delta$, and $S^*x^* =x^*$ for all $x^*\in F={\rm ran}~P^*$. Similarly to (a), we get that $PS=P$. Hence, $Q:=P+S-SP\in\mathcal{F}(X)$ is a needed projection.
\end{proof}

An impact of the Radon--Nikod\'{y}m property can improve Theorem 4.4 as follows. This is clear, because under the assumptions below, by \cite[Corollary 5.12]{LiO}, the pair $(X^*, Y^\bot)$ has the metric (i.e., 1-bounded) approximation property whenever $(X^*, Y^\bot)$ has the {\it approximation property}. The latter property means that, in the definition of the $\lambda$-bounded approximation property of a pair, one does not put any restriction on the norms of the approximating operators $S$ (i.e., one deletes ``and $\| S\|\leq\lambda+\varepsilon$").

\begin{cor}
Let $X$ be a Banach space and let $Y$ be a closed subspace of $X$. Let $X^*$ or $X^{**}$ have the Radon--Nikod\'{y}m property.

{\rm (a)} If $(X^* ,Y^\bot)$ has the $\pi_\lambda$-property, then the pair $(X,Y)$ has the $\pi_{2\lambda+1}$-duality property.

{\rm (b)} If $(X,Y)$ has the $\pi_\lambda$-property and $(X^*, Y^\bot)$ has the approximation property, then the pair $(X,Y)$ has the $\pi_{2\lambda+1}$-duality property.
\end{cor}

Is the assumption about the Radon--Nikod\'{y}m property in Corollary 4.5 essential or not? This question is closely connected with the following famous open problem (see, e.g., \cite[Problem 3.8]{C}; for an overview around the problem, see \cite[Section 3]{O2}). Does the approximation property of the dual space $X^*$ of an arbitrary Banach space $X$ imply the metric approximation property?

Finally, let us look at
the situation when $X^{**}$ or, more generally, the pair $(X^{**}, Y^{\bot\bot})$ has the $\pi_\lambda$-property. Here, applying Theorem 4.4 (a) twice, we would get that $(X,Y)$ has the $\pi_\mu$-property with (a rather big) $\mu=(\lambda^2+2\lambda)^2+2(\lambda^2+2\lambda)$. However, in this case, the situation turns out to be perfect: $\mu=\lambda$, as the following result shows. Its
special case with $Y=X$ (or $Y=\{0\}$) is also due to Johnson, Rosenthal,
and Zippin \cite[Corollary 3.4]{JRZ}, and Theorem 1.2  enables us to give an easy proof of it.

\begin{cor}
Let $X$ be Banach space and let $Y$ be a closed subspace of $X$. If the pair  $(X^{**}, Y^{\bot\bot})$ has the $\pi_\lambda$-property, then also  the pair $(X,Y)$ has the $\pi_\lambda$-property.
\end{cor}
\begin{proof}
Let $E$ be a finite-dimensional subspace of $X$ and let $\eps>0$.
Since $(X^{**}, Y^{\bot\bot})$ has the $\pi_\lambda$-property, there exists a projection $P\in\mathcal{F}(X^{**})$ such that $P(Y^{\bot\bot})\subset Y^{\bot\bot}$,
$\|P\|\leq\lambda+\eps/2$, and $Px=x$ for all $x\in E$.
By Theorem 1.2, there is a projection $Q\in\mathcal{F}(X)$ such that $Q(Y)\subset Y$,  $\|Q\|\leq\lambda+\eps$,
and $Qx=\Qss x=Px=x$ for all $x\in E$ (note that then $Px=x\in X$).
\end{proof}

\end{document}